\documentclass[12pt,a4paper]{amsart}

\usepackage{amsmath}
\usepackage{amssymb}

\usepackage{amscd}
\usepackage{epsfig}

\usepackage{amsthm,amssymb,amsmath,enumitem, verbatim, color, caption, faktor,esint}

\usepackage{fullpage}
\usepackage[utf8]{inputenc}
\usepackage[T1]{fontenc}
\usepackage[all]{xy}

\usepackage{hyperref}
\usepackage{filecontents}

\newcommand{\N}{\mathbb N}
\newcommand{\R}{\mathbb R}

\newcommand{\mc}{\mathcal}

\newcommand{\id}{\operatorname{id}}

\newcommand{\inj}{\operatorname{inj}}
\newcommand{\grad}{\operatorname{grad}}
\newcommand{\divgc}{\operatorname{div}}

\newcommand{\Ric}{\operatorname{Ric}}

\newcommand{\diam}{\operatorname{diam}}

\newcommand{\Diff}{\operatorname{Diff}}
\newcommand{\vol}{\operatorname{vol}}
\newcommand{\tr}{\operatorname{tr}}

\newcommand{\Vol}{\operatorname{Vol}}
\newcommand{\CI}{\mc{C}^{\infty}}
\newcommand{\sym}{\operatorname{sym}}

\newcommand\halfopen[2]{\ensuremath{[#1,#2)}}
\newcommand\halfclosed[2]{\ensuremath{(#1,#2]}}

\usepackage{enumitem}

\newtheoremstyle{break}% name
  {}%         Space above, empty = `usual value'
  {}%         Space below
  {\itshape}% Body font
  {}%         Indent amount (empty = no indent, \parindent = para indent)
  {\bfseries}% Thm head font
  {.}%        Punctuation after thm head
  {\newline}% Space after thm head: \newline = linebreak
  {}%         Thm head spec

\theoremstyle{break}

\newtheorem{thm}{Theorem}[section]
\newtheorem{prop}{Proposition}[section]

\newtheorem{lemma}{Lemma}[section]
\newtheorem*{ithm}{Theorem}

\theoremstyle{definition}
\newtheorem{defn}{Definition}[section]

\begin{document}
\title{Blow up criteria for geometric flows on surfaces}
\author{Lothar Schiemanowski}
\address{Christian-Albrechts-Universität zu Kiel, Mathematisches Seminar, Ludewig-Meyn-Straße 4, Kiel}
\email{schiemanowski@math.uni-kiel.de}
\begin{abstract}
Utilizing a splitting of geometric flows on surfaces introduced by Buzano and Rupflin, we present a general scheme to prove blow up criteria for such geometric flows. A vital ingredient is a new compactness theorem for families of metrics on surfaces with a uniform bound on their volumes, square integrals of their curvatures and injectivity radii.  In particular we prove blow up criteria for the harmonic Ricci flow and for the spinor flow on surfaces. 
\end{abstract}
\maketitle

\section{Introduction}
Let $M$ be a closed orientable surface with non-positive Euler characteristic $\chi(M)$. Given an initial metric $g$ on $M$ and perhaps some other piece of data $s$ (for example a map, a connection, a section of a vector bundle, ...), a geometric flow associates to $(g, s)$ a family $(g_t, s_t)$, such that $(g_0, s_0) = (g,s)$. Typically, the geometric flow is given as a solution of a nonlinear partial differential equation. The most basic result to be established for geometric flows is short time existence and uniqueness: for a class of admissible initial data $(g,s)$ there exists a unique maximal solution $(g_t, s_t)$ with $(g_0, s_0) = (g,s)$ on an interval $\halfopen{0}{T_{\max}}$. If $T_{\max} < \infty$, the flow becomes singular in some way. It is of great interest to understand what happens to the metric and the data as we approach the singular time $T_{\max}$. The goal of this article is to present a fairly general way of obtaining conditions on $(g_t, s_t)$ on an interval $\halfopen{0}{T}$, which exclude the formation of singularities at time $T$. The strategy for this relies on a detailed understanding of the space of metrics on closed surfaces.

The uniformization theorem says that for any metric $g$ on $M$ there exists a unique metric $\bar g$ of constant curvature with volume $1$, which is conformal to $g$, i.e. which satisfies
$$\bar g = e^{2u} g$$
for some $u \in \CI(M)$. Introducing the space of Riemannian metrics on $M$
$$\mc{M} = \{g \text{ Riemannian metric on } M\}$$
and its subspace of constant curvature metrics (of unit volume)
$$\mc{M}^{cc} = \{g \text{ Riemannian metric with constant curvature and } \Vol(M,g) = 1\},$$
the uniformization theorem can then be reformulated as the statement that
$$\CI(M) \times \mc{M}^{cc} \to \mc{M}$$
$$(u, \bar g) \mapsto e^{2u} \bar g$$
is a bijection. The group of diffeomorphisms $\Diff(M)$ acts on $\mc{M}$ and also on $\mc{M}^{cc}$ by pullback. The orbits of $\Diff(M)$ in $\mc{M}$ correspond to isometry classes of Riemannian metrics on $M$. It turns out that the quotient $\mc{S} = \mc{M}^{cc} / \Diff(M)$, i.e. the set of isometry classes of constant curvature metrics on $M$, is in a natural way a finite dimensional orbifold.

Now suppose $g_t$ is a smooth family of metrics parametrized by an interval $(0,T)$. By the uniformization theorem there exists a family of conformal factors $u_t$ and a family of constant curvature metrics $\bar g_t \in \mc{M}^{cc}$, such that
$$g_t = e^{2u_t} \bar g_t.$$
To formulate useful blow up criteria, we thus need to understand under which geometric conditions the conformal factors and the constant curvature factors remain controlled in a useful way. Since the constant curvature condition is diffeomorphism invariant, a purely geometric (i.e. isometry invariant condition) can not suffice to control the constant curvature metrics $\bar g_t$. This issue can be overcome by changing the family $\bar g_t$ by a family of diffeomorphisms. The following theorem is the foundation for the blow up criteria we prove later.
\begin{ithm}
  \label{GeomCtrl}
  Suppose $M$ is a closed surface with $\chi(M) \leq 0$ and let $\hat g$ be any Riemannian metric on $M$. Suppose $g_t \in \mc{M}$ is a smooth family defined on an interval $(0,T)$ and suppose
  $$\sup_{0 < t < T} \Vol(M, g_t) < \infty,$$
  $$\sup_{0 < t < T} \int_M |R_{g_t}|^2 \vol_{g_t} < \infty$$
  and
  $$\inf_{0 < t < T} \inj(M, g_t) > 0.$$
  Then there exist a family of diffeomorphisms $f_t$, a family of constant curvature metrics $\bar g_t$ and a family of conformal factors $u_t \in \CI(M)$ 
  $$g_t = f_t^*(e^{2u_t} \bar g_t)$$
  and
  $$\sup_{0 < t < T} \|u_t\|_{H^2(M, \hat g)} < \infty$$
  and
  $$\inf_{0 < t < T} \inj(M,g_t) > 0.$$
\end{ithm}
Thus, assuming uniform control on the volumes, the $L^2$ norms of the curvatures and on the injectivity radius, we obtain uniform control on the conformal factors in $H^2$ and on the constant curvature metrics in any $C^k$ norm - provided we pull back the metrics by an adequate family of diffeomorphisms.

The theorem suggests that a natural assumption to rule out blow up along a geometric flow is that the three quantities in the theorem remain bounded. This assumption may not be sufficient, since the extra piece of data $A_t$ may blow up independently from the geometry. For many geometric flows the coupled data satisfies a partial differential equation which has been studied independently on a fixed Riemannian manifold. The control on the geometry given by theorem \ref{GeomCtrl} is typically enough to apply these results with minimal modifications.

The above strategy will be applied in two cases: the harmonic Ricci flow and the spinor flow. We will recall the precise definitions of these flows in section \ref{blowup}. In the following theorems we assume that $M$ is a closed surface with $\chi(M) \leq 0$.
\begin{ithm}
  Suppose $(g_t, f_t)$ is a solution of the harmonic Ricci flow on $M \times \halfopen{0}{T}$ with
  $$\sup_{0 \leq t < T} \int_M |R_{g_t}|^{2+\epsilon} + |df_t|^{4+2\epsilon} \vol_{g_t} < \infty$$
  for some $\epsilon > 0$ and
  $$\inf_{0 \leq t < T} \inj(M, g_t) > 0.$$
  Then the solution can be extended smoothly to a larger interval $\halfopen{0}{\tilde T}$, $\tilde T > T$.
\end{ithm}
\begin{comment}
\begin{thm}
  Suppose $(g_t, A_t)$ is a solution of the Ricci Yang--Mills flow on $M$ with
  $$\sup_{0 < t < T} \int_M |R_{g_t}|^2 + |F_{A_t}|^2 \vol_{g_t} < \infty$$
  and
  $$\inf_{0 < t < T} \inj(M, g_t) > 0.$$
  Then the solution can be extended to a larger interval $(0,\tilde T)$, $\tilde T$.
\end{thm}
\end{comment}
\begin{ithm}\label{IntBC}
  Suppose $(g_t, \varphi_t)$ is a solution of the spinor flow on $M \times \halfopen{0}{T}$ with
  $$\sup_{0 \leq t < T} \int_M |\nabla^2 \varphi_t|^q \vol_{g_t} < \infty$$
  for some $q > 4$ and
  $$\inf_{0 \leq t < T} \inj(M, g_t) > 0.$$
  Then the solution can be extended smoothly to a larger interval $\halfopen{0}{\tilde T}$, $\tilde T > T$.
\end{ithm}
In the case of the spinor flow, even the following pointwise blowup criterium is new.
\begin{ithm}\label{UniBC}
  Suppose $(g_t, \varphi_t)$ is a solution of the spinor flow on $M \times \halfopen{0}{T}$ with
  $$\sup_{\substack{x \in M \\ 0 \leq t < T}}|\nabla^2 \varphi_t(x)| < \infty.$$
  Then the solution can be extended smoothly to a larger interval $\halfopen{0}{\tilde T}$, $\tilde T > T$.
\end{ithm}
The paper is structured as follows. After recalling some facts related to parabolic equations in section \ref{parabolic}, we recall the splitting of geometric flows on surfaces into conformal factors and horizontal curves of constant curvature metrics introduced by Buzano and Rupflin in \cite{Buzano2017}, minorly generalizing the setting in section \ref{splitting}. Then we prove theorem \ref{GeomCtrl} in section \ref{GC}. Finally, in section \ref{blowup} we turn to the blowup criteria themselves.

\section*{Acknowledgements}
The author thanks H. Weiß for encouragement and many discussions. This work was completed during a visit at Queen Mary University London. The author gratefully acknowledges financial support of the DAAD during this period. Furthermore, he thanks Queen Mary University and its geometry and analysis group for the hospitality and his host R. Buzano for valuable feedback regarding this work.

\section{Parabolic function spaces and estimates}
\label{parabolic}
We introduce some notation relating to parabolic equations. Let $M$ be a closed Riemannian manifold, $I \subset \R$ some interval. We denote by
$$W^{2,1}_p(M \times I)$$
the completion of the space of time dependent functions $\CI(M \times I)$ with the norm
$$\|f\|_{W^{2,1}_p(M \times I)}^p = \int_I \int_M |\partial_t f_t|^p + |\nabla f_t|^p + |\nabla^2 f_t|^p \vol_g dt.$$
If $M$ and $I$ are clear from the context, we will just write $W^{2,1}_p$ instead. The significance of this space is that for a parabolic equation
$$\partial_t u_t + L_t u_t = f_t$$
there is the estimate
$$\|u_t\|_{W^{2,1}_p (M \times I)} \leq C \left( \|u\|_{L^p(M \times I)} + \|f\|_{L^p(M \times I)} \right).$$
On $M \times I$ we introduce the parabolic distance
$$d^p((x_1,t_1), (x_2, t_2)) = \left( d_g(x_1, x_2)^2 + |t_1 - t_2| \right)^{1/2}.$$
With respect to this distance we can define the Hölder constant in the usual fashion for any $\alpha \in (0,1)$. For a function $f$ on $M \times I$ we denote its Hölder constant with respect to $d^p$ by $[f]_{\alpha, \alpha/2}$. The space of $\alpha,\alpha/2$-Hölder continuous functions is denoted by
$$C^{\alpha,\alpha/2}(M \times I)$$
and is given the norm
$$\|f\|_{C^{\alpha,\alpha/2}} = \sup_{x \in M} |f(x)| + [f]_{\alpha,\alpha/2}.$$
We also introduce the space of functions whose first $k$ {\em spatial} derivatives are $\alpha,\alpha/2$ Hölder continuous
$$C_k^{\alpha,\alpha/2}(M \times I) = \{f \in C^{\alpha, \alpha/2}(M \times I) : \nabla^i f \in C^{\alpha, \alpha/2} \text{ for all } i \leq k\}$$
with the obvious norm and the space of functions whose first second spatial derivatives and first temporal derivative are $\alpha, \alpha/2$ Hölder continuous
$$C^{2+\alpha,1+\alpha/2}(M \times I) = \{f \in C^{\alpha, \alpha/2}(M \times I) : \nabla f, \nabla^2 f, \partial_t f \in C^{\alpha, \alpha/2} \}.$$
The significance of this space is that for a parabolic equation
$$\partial_t u_t + L_t u_t = f_t$$
there is the Schauder estimate
$$\|u_t\|_{C^{2+\alpha, 1+\alpha/2}} \leq C \left( \|u\|_{C^{\alpha, \alpha/2}} + \|f\|_{C^{\alpha,\alpha/2}} \right).$$
Clearly, there is an embedding
$$C^{2+\alpha, 1+\alpha/2} \hookrightarrow C_2^{\alpha,\alpha/2}.$$
Moreover, there is a parabolic Sobolev embedding
$$W^{2,1}_p (M \times I) \hookrightarrow L^q(M \times I)$$
with
$$q = \frac{(n+2) p}{n+2-2p}.$$
If $p > (n+2)/2$, we get embeddings
$$W^{2,1}_p (M \times I) \hookrightarrow C^{\alpha,\alpha/2}(M \times I)$$
for all
$$\alpha < 2 - \frac{n+2}{p}.$$
If $p > n+2$, we get an embedding
$$W^{2,1}_p (M \times I) \hookrightarrow C_1^{\beta, \beta/2}(M \times I)$$
for all
$$\beta < 1 - \frac{n+2}{p}.$$
These embeddings can be found in \cite{Ladyzenskaja1995}, Lemma II.4.3.

\section{Decomposition of flows}
\label{splitting}
Let $g_t$ be a time-dependent family of Riemannian metrics on a closed surface $M$ with $\chi(M) \leq 0$. The uniformization theorem tells us that there is a unique family $u_t$ of smooth functions, such that
$$\hat g_t = e^{-2u_t} g_t$$
are metrics of constant curvature with volume $1$. By pulling back the family $\hat g_t$ by a family of diffeomorphisms $f_t$, we can arrange for $\bar g_t = f_t^* \hat g_t$ that
$$\partial_t (\bar g_t) \in \mc{H}_{\bar g_t} \subset T_{\bar g_t} \mc{M},$$
where $\mc{H}_g$ is the orthogonal complement of the (infinitesimal) orbit of the diffeomorphism group through $g \in \mc{M}^{cc}$. It turns out that, as Rupflin and Topping observe in \cite{Rupflin2016}, provided the injectivity radii of $\bar g_t$ are bounded from below and the velocity of the curve is bounded above in $L^2$, we get very good control on the family $\bar g_t$.

We will use this observation to split a geometric flow into a family of conformal factors and a family of constant curvature metrics, which satisfy the above condition. We will obtain new evolution equations for the flow. This strategy was used by Buzano and Rupflin to study the harmonic Ricci flow in \cite{Buzano2017}. We will very slightly generalize their results. Where they assumed that the geometric flow consists of an evolving metric and a map from the surface into some fixed manifold, we will instead assume that the geometric flow consists of an evolving metric and a section of a fiber bundle, allowing for different transformation behaviors under diffeomorphisms.

Since we are interested in geometric flows which are coupled to some additional data, we will define the notion of a {\em coupled geometric flow} and we will then show how to split such a flow as indicated above.  We then derive the evolution equations satisfied by this split flow.

At this point we also make a remark regarding notation in this and the following sections: for evolution equations we will often drop the time subscript to make them more legible.

Let $E$ be a fiber bundle over $M$. We require that there exists a pullback operation, at least for diffeomorphisms, i.e. given a diffeomorphism $f \in \Diff(M)$ and a section $s \in \Gamma(E)$, there should exist $f^*s$. Furthermore we ask that there exists a connection in the sense that we can differentiate families of sections, i.e. given a family of sections $s_t$, there exists a notion of time derivative
$$\partial_t|_{t=0} s_t \in T_{s_0} \Gamma(E).$$
\begin{defn}
  \label{CGF}
  Suppose $Q : \mc{M} \times \Gamma(E) \to T\mc{M} \times T\Gamma(E)$ is a diffeomorphism invariant vector field, i.e.
  $$f^*Q(g, s) = Q(f^*g, f^*s).$$
  We say that $Q$ defines a {\em coupled geometric flow}. A family $(g_t, s_t) \in \mc{M} \times \Gamma(E)$, which satisfies
  $$\partial_t (g_t, s_t) = Q(g_t, s_t)$$
  is a solution of this coupled geometric flow.
\end{defn}
Notice that in this setting we have a Lie derivative for sections of the fiber bundle $E$:
$$\mc{L}^E_X s = \partial_t|_{t=0} f_t^*s,$$
where $s \in \Gamma(E)$ and $f_t$ is the flow of the vector field $X$.

We will now introduce the corresponding split flow. Given a constant curvature metric $g$, the tangent space of $\mc{M}^{cc}$ intersected with the space orthogonal to the diffeomorphism orbit through $g$ is given by
$$\mc{H}_g = \{ h \in \Gamma( \odot^2 T^* M) : \delta_g h = 0, \tr_g h = 0\}.$$
The following result from \cite{Buzano2017}  provides that a family of constant curvature metrics can be pulled back to be tangent to $\mc{H}_g$. This can be considered to be a canonical gauge for the family.
\begin{lemma}[Lemma 2.2,\cite{Buzano2017}]
  \label{Pullback}
  Given a family of constant curvature metrics $\bar g_t \in \mc{M}^{cc}$, there exists a unique family of diffeomorphisms $f_t \in \mc{M}^{cc}$, such that
  $$\partial_t (f_t^* \bar g_t) \in \mc{H}_{\bar g_t}$$
  and $f_0 = \id_M$.
\end{lemma}
Now suppose $(\tilde g_t, \tilde s_t)$ is a solution of the coupled geometric flow given by a vector field $Q$. First we apply the uniformization theorem to $\tilde g_t$ to obtain a unique conformal metric $\hat g_t = e^{-2\hat u_t} \tilde g_t$. Then, applying the lemma to the family $\hat g_t$, we obtain a family of diffeomorphisms $f_t$, such that $\bar g_t = f_t^* \hat g_t$ fulfills the above relation, i.e. $\partial_t \bar g_t \in \mc{H}_{\bar g_t}$. There is a slight subtlety here in the case of $\chi(M) = 0$. This is because the flat torus has Killing vector fields. At any time $t$ there is a choice involved when defining the diffeomorphisms $f_t$, which arise from a time dependent family of vector fields. This issue can be overcome by introducing a normalization condition as explained in remark 2.1 in \cite{Buzano2017}.

We then define
$$g_t = f_t^* \tilde g_t = e^{2 u_t} \bar g_t,$$
where $u_t = \hat u_t \circ f_t$. We also define $s_t = f_t^* s_t$.
\begin{defn}
  \label{SF}
  The family $(\bar g_t, u_t, s_t, f_t)$ is the {\em split flow} corresponding to the flow $(\tilde g_t, \tilde s_t)$.
\end{defn}
We will now derive the evolution equations of $(\bar g_t, u_t, s_t)$. Let $X_t$ be the time dependent vector field generating the family of diffeomorphisms $f_t$.
Notice that we can split the vector field $Q$ into components $Q = (Q_m, Q_E)$.
We start with the evolution of $\bar g_t$. We have
\begin{align*}
  \partial_t \bar g_t & = \partial_t (e^{-2u_t} f_t^* \tilde g_t) \\
  & = -2 e^{-2u_t} (\partial_t u_t) e^{2u_t} \bar g_t + e^{-2u_t} \mc{L}_{X_t} g_t + e^{-2u_t} f_t^* \partial_t \tilde g_t \\
  & = -2 (\partial_t u_t) \bar g_t + \mc{L}_{X_t} \bar g_t + 2 (X_t u_t) \bar g + e^{-2u_t} f_t^* Q_m(\tilde g_t, \tilde s_t) \\
  & = (-2 \partial_t u_t + 2 X_t u_t) \bar g + \mc{L}_{X_t} \bar g_t + e^{-2u_t} Q_m(g_t, s_t) \\
  & = \left(-2 \partial_t u_t + 2 X_t u_t + \frac{1}{2} \tr_{g_t} Q_m(g_t, s_t)\right) \bar g_t + \mc{L}_{X_t} \bar g_t + e^{-2u_t} Q_m(g_t, s_t)\\
  & = \rho_t \bar g_t + \mc{L}_{X_t} \bar g_t + e^{-2u_t} \mathring{Q}_m(g_t, s_t)
\end{align*}
where
$$\rho_t = -2 \partial_t u_t + 2 X_t u_t + \frac{1}{2} \tr_{g_t} Q_m(g_t, s_t)$$
is a time dependent function on $M$. Rewriting this equation also yields
$$\partial_t u = \frac{1}{4} \tr_g Q_m(g, s) + Xu - \frac{1}{2} \rho.$$
It turns out that $\rho$ satisfies an elliptic equation on every time slice $M \times \{t\}$, which we will derive now. Observe that $\partial_t \bar g_t$ is a variation that preserves the constant curvature condition. Since the constant curvature condition is diffeomorphism invariant, the variation
$$\rho_t \bar g_t + e^{-2u_t} \mathring{Q}_m(g_t, s_t)$$
also preserves constant curvature. Let $h = \rho_t \bar g_t + e^{-2u_t} \mathring{Q}_m(g_t, s_t)$. Recall the formula
$$\partial_t|_{t=0} R_{\bar g+th} = -\frac{1}{2} R_{\bar g} \tr_{\bar g} h - \Delta_{\bar g} \tr_{\bar g} h + \delta_{\bar g} \delta_{\bar g} h.$$
For our variation $h$ we thus obtain
$$0 = - R_{\bar g} \rho - 2 \Delta_{\bar g} \rho + \Delta_{\bar g} \rho + \delta_{\bar g} \delta_{\bar g} (e^{-2u} \mathring{Q}_m(g, s)).$$ 
Thus $\rho$ fulfills the equation
$$\Delta_{\bar g} \rho + R_{\bar g} \rho - \delta_{\bar g} \delta_{\bar g} (e^{-2u} \mathring{Q}_m(g,s)) = 0$$
at any time $t$. The relation $\partial_t \bar g_t \in \mc{H}_{\bar g_t}$ implies in particular $\delta_{\bar g_t} \partial_t \bar g_t = 0$. Using that $\delta_g^* X^{\flat} = \mc{L}_X g$, we find
$$\delta_{\bar g} \delta_{\bar g}^* X^{\flat} = -\delta_{\bar g}(e^{-2u}(\mathring{Q}_m(g, \varphi) + \rho g_0),$$
for every time $t$. This is also an elliptic equation for $X$. The normalization condition from remark 2.1 in \cite{Buzano2017} ensures a unique solution to this equation.

Let
$$P_{\bar g} : \Gamma(\odot^2 T^*M) \to \mc{H}_{\bar g}$$
be the orthogonal projection. The equation for $\partial_t \bar g$ implies
$$\partial_t \bar g = P_{\bar g}(e^{-2u} Q_m(g, s)).$$
Finally, the evolution of $s_t$ is given by
\begin{align*}
  \partial_t s_t & = \partial_t (f_t^* \tilde s_t) \\
  & = f_t^* (\partial_t \tilde s_t + \mc{L}_X^F \tilde s_t) \\
  & = f_t^* (Q_E(\tilde g_t, \tilde s_t) + \mc{L}_{X_t}^F \tilde s_t) \\
  & = Q_E(g_t, s_t) + \mc{L}_X^F s_t.
\end{align*}
We sum up these results in the following proposition. Compare \cite{Buzano2017}, prop 2.3.
\begin{prop}
  \label{SFE}
  Suppose $(\tilde g_t, \tilde s_t)$ is a coupled geometric flow in the sense of definition \ref{CGF}. Let $(\bar g_t, u_t, s_t)$ be the split flow of $(\tilde g_t, \tilde s_t)$ in the sense of definition \ref{SF} and let $f_t$ be the corresponding family of diffeomorphisms. Let $X_t$ be the generating vector field of $f_t$. The split flow satisfies the following equations:
  \begin{align}
    \partial_t \bar g & = P_{\bar g}(e^{-2u} Q_m(g, s))\\
    \partial_t u & = \frac{1}{4} \tr_g Q_m(g, s) + Xu - \frac{1}{2} \rho\\
    \partial_t \varphi & = Q_E(g, s) + \mc{L}_X^F s
  \end{align}
  where $P_{\bar g}: \Gamma(\odot^2 T^*M) \to \mc{H}_{\bar g}$ is the orthogonal projection and $\rho$ is the unique solution of
  \begin{equation}
    \label{RhoDef}
    \Delta_{\bar g} \rho + R_{\bar g} \rho - \delta_{\bar g} \delta_{\bar g} (e^{-2u} \mathring{Q}_m(g,s)) = 0.
  \end{equation}
  Furthermore $X$ solves
  \begin{equation}
    \label{XDef}
    \delta_{\bar g} \delta_{\bar g}^* X^{\flat} = -\delta_{\bar g}(e^{-2u}(\mathring{Q}_m(g, s) + \rho \bar g).
  \end{equation}
\end{prop}

We have the following straightforward estimates for $X$ and $\rho$. The proposition is essentially lemma 2.5 in \cite{Buzano2017}, slightly adapted to our case.
\begin{prop}
  \label{XRhoEst}
  If $\rho$ solves equation \ref{RhoDef}, then
  $$\|\rho\|_{L^p(M, \bar g)} \leq C(\inj(M, \bar g)) \|e^{-2u} \mathring{Q}_m(g, s)\|_{L^p(M, \bar g)} \leq C(\inj(M,\bar g), \|u\|_{C^0}) \|Q_m(g, s)\|_{L^p(M, g)}.$$
  If $X$ solves equation \ref{XDef}, then
  $$\|X\|_{W^{1,p}(M, \bar g)} \leq C(\inj(M, \bar g)) \|e^{-2u} Q_m(g, s)\|_{L^p(M, \bar g)} \leq C(\inj(M,\bar g), \|u\|_{C^0}) \|Q_m(g, s)\|_{L^p(M, g)}.$$
\end{prop}
\begin{proof}
  Observe that
  $$\delta_{\bar g}: L^p \to W^{-1, p}$$
  and
  $$\delta_{\bar g} \delta_{\bar g}^*: L^p \to W^{-2,p}$$
  are continuous. Furthermore, $\Delta_{\bar g} + R_{\bar g} : L^p \to W^{-2,p}$ is invertible on the relevant space of functions. Similarly, $\delta_g \delta_g^* : W^{1,p} \to W^{-1, p}$ is invertible on the relevant space of vector fields. By the Mumford compactness criterium, the constants of these operators only depend on the injectivity radius of $\bar g$. These facts yield the inequality
  $$\|\rho\|_{L^p(M, \bar g)} \leq C(\inj(M, \bar g)) \|e^{-2u} \mathring{Q}_m(g, \varphi)\|_{L^p(M, \bar g)}$$
  and its vector field analogue. The inequality
  $$C(\inj(M, \bar g)) \|e^{-2u} \mathring{Q}_m(g, \varphi)\|_{L^p(M, \bar g)} \leq C(\inj(M,\bar g), \|u\|_{C^0}) \|Q_m(g, \varphi)\|_{L^p(M, g)}$$
  and its vector field analogue follows by expressing the volume element of $g$ in terms of the volume element of $\bar g$ and estimating $e^{2u}$ by $e^{2\|u\|_{C^0}}$.
\end{proof}

\section{A compactness theorem for metrics with a bound on the injectivity radius and square integral of the curvature}
\label{GC}
Suppose $M$ is a closed surface with $\chi(M) \leq 0$. In this section we will prove the theorem \ref{GeomCtrl}. Its proof will be based on the following apriori estimate for solutions of the constant curvature equation.
\begin{thm}
  \label{APE}
  Suppose $M$ is a closed surface and suppose $\chi(M) \leq 0$. Let $g$ be any Riemannian metric and let $\bar g = e^{2u} g$ be the unique metric of constant curvature and volume $1$ conformal to $g$.
  Then $\sup_{x \in M} |u|$ can be bounded by a constant depending only on the injectivity radius $\inj(M,g)$, the volume $\Vol(M,g)$ and $\int_M |R_g|^2 \vol_g$.
\end{thm}
The proof is based on the Ricci flow, which solves the constant curvature equation on surfaces in the limit $t \to \infty$. Let $g$ be any metric on a closed surface $M$ with $\chi(M) \leq 0$. Consider the normalized Ricci flow $g_t$ with initial condition $g$, i.e. the family of metrics $g_t$ satisfying
$$\partial_t g_t = (r - R_{g_t}) g_t$$
$$g_0 = g,$$
where
$$r = \Vol(M,g)^{-1} \int_M R_{g} \vol_g =  \Vol(M,g_t)^{-1} \int_M R_{g_t} \vol_{g_t}.$$
It is well known that in this case the normalized Ricci flow exists for all times and that it converges to the metric of constant curvature with equal volume. Since all the metrics $g_t$ are conformal, we can write them as $g_t = e^{2u_t} g$. We denote the limit of the metrics $g_t$ as $t\to \infty$ by $g_{\infty} = e^{2u_{\infty}} g$. Our goal is to establish an estimate on $u_{\infty}$. We may as well establish bounds on $u_t$ independent of time. To do this, we proceed in two steps. First we control $g_t$ on some, perhaps small, time interval $[0,T]$. At the time $T$ we will have bounds on the maximal curvature, rather than only on its $L^2$ norm. With these bounds we then obtain bounds for all future times $\halfopen{T}{\infty}$.
These estimates essentially follow from the following three theorems.
\begin{thm}[Yang, \cite{Yang1992}]
  Given a Riemannian metric $g$ on a surface with Sobolev constant $C_S(M,g) = \sigma$, $K = \int_M R_g^2 \vol_g / \Vol(g)$, then considering the Ricci flow $g_t$ with initial condition $g_0 = g$, there exists a $T > 0$
  with
  $$C_S(M,g_T) \geq \sigma/2$$
  $$\int_M R_{g_T}^2 \vol_{g_T} / \Vol(g_T) \leq 2 K$$
  $$|R_{g_T}| \leq C_1(\sigma, K)$$
  $$\|u\|_{C^0} \leq C_2(\sigma, K)$$
\end{thm}
\begin{thm}[Hamilton, Chow et al]
  \label{LTE}
  Suppose $M$ is a closed surface with $\chi(M) \leq 0$. Suppose $g_t$ is the normalized Ricci flow on $M$ with initial condition $g_0 = g$. Then there exists a constant $C > 1$, such that
  $$C^{-1} g \leq g_t \leq C g.$$
  The constant $C$ depends only on $\max |f|$, where $f$ is the curvature potential of $g$, i.e. the unique function satisfying
  $$\Delta_g f = R_g - r \text{ and } \int_M f \vol_g = 0.$$
\end{thm}
This result goes back to Hamilton \cite{Hamilton1988}, but the version we cite can be found in \cite{Chow2004}, lemma 5.12 and corollary 5.15.
\begin{thm}[Calderon, Wang \cite{Wang2004}]
  \label{CI}
  Let $M$ be closed surface. For any metric $g$ on $M$ there exists a constant $C> 0$ depending only on $\inj(M,g), \Vol(M,g)$ and $\sup_{x \in M} |R_g(x)|$, such that for $u \in \CI(M)$ with $\int_M u \vol_g = 0$ we have
  $$\|u\|_{H^2(M,g)} \leq C \|\Delta_g u\|_{L^2}.$$
\end{thm}
The estimate in the previous theorem is called Calderon inequality. That the constant depends only on the quantities listed in the theorem can be found in \cite{Wang2004}.
\begin{proof}[Proof of theorem \ref{APE}]
  We denote by
  $$\mc{R}^2(M, g) = \int_M |R_g|^2 \vol_g.$$
  Let $g_t = e^{2u_t} g$ be the normalized Ricci flow with initial condition $g_0 = g$. Notice that the Sobolev constant of $(M,g)$ can be estimated from below solely in terms of $\Vol(M,g)$ and $\inj(M,g)$ by a theorem of Croke, \cite{Croke1980}. Furthermore, a lower bound on the injectivity radius also implies a lower bound on the volume on a surface, because a lower bound on the Sobolev constant implies a lower bound on the volume of balls. Hence Yang's theorem applies and we obtain a time $T> 0$, such that
  $$\sup_{x \in M} |R_{g_T}(x)| < C_1(\inj(M,g), \mc{R}^2(M,g)),$$
  $$C_S(M, g_T) > C_S(M,g)/2$$
  and
  $$\|u_T\|_{C^0} \leq C_2(\inj(M,g), \mc{R}^2(M,g)).$$
  Since the injectivity radius can be bounded from below by the Sobolev constant, the maximal curvature and the volume, we also obtain
  $$\inj(M,g_T) \geq C_3(\inj(M,g), \mc{R}^2(M,g), \Vol(M,g)).$$
  Now consider the curvature potential of $g_T$, i.e. the function $f$ satisfying
  $$\Delta_{g_T} f = R_{g_T} \text{ and } \int_M f \vol_g = 0.$$
  We have
  $$\|R_{g_T}\|_{L^2}^2 = \int_M |R_{g_T}|^2 \vol_{g_T} \leq C_1^2 \int_M e^{2u_T} \vol_g \leq C_1^2 e^{2 C_2} \Vol(M,g),$$
  in particular there exists $C_4(\inj(M,g), \mc{R}^2(M,g), \Vol(M,g)) > 0$, such that
  $$\|R_{g_T}\|_{L^2} \leq C_4.$$
  By the Calderon inequality we have
  $$\|f\|_{H^2(M, g_T)} \leq C(\inj(M,g_T), \Vol(M,g_T), \sup_{x \in M} |R_{g_T}(x)|) \|R_{g_T}\|_{L^2}.$$
  Combining all our previous estimates, we thus get
  $$\|f\|_{H^2(M, g_T)} \leq C_5(\inj(M,g), \mc{R}^2(M,g), \Vol(M,g)).$$
  Since we have already bounded the Sobolev constant of $(M,g_T)$ in terms of the Sobolev constant of $(M,g)$, we have
  $$\|f\|_{C^0} = \max |f| \leq C_6(\inj(M,g), \mc{R}^2(M,g), \Vol(M,g))$$
  by Sobolev embedding $H^2 \hookrightarrow C^0$. Finally, it follows from theorem \ref{LTE}, that there exists a constant $C > 1$ depending only on $\max |f|$ and thus only on $\inj(M,g), \mc{R}^2(M,g)$ and $\Vol(M,g)$, such that
  $$C^{-1} g_T \leq g_t \leq C g_T$$
  for all $t \geq T$. Since $g_t = e^{2u_t} g = e^{2(u_t - u_T)} g_T$, this implies
  $$\|u_t - u_T\|_{C^0} \leq \frac{1}{2} |\log C|.$$
  We already have a bound on $\|u\|_{C^0}$ and thus we obtain
  $$\|u_t\|_{C^0} \leq C(\inj(M,g), \mc{R}^2(M,g), \Vol(M,g))$$
  independent of $t$. This implies in particular
  $$\|u_{\infty}\|_{C^0} \leq C(\inj(M,g), \mc{R}^2(M,g), \Vol(M,g)),$$
  proving the theorem.
\end{proof}
We will need the following lemma.
\begin{lemma}
  Suppose $M$ is a closed surface. Then for any Riemannian metric $g$, the diameter can be bounded in terms of the injectivity radius $\inj(M,g)$ and the volume $\Vol(M,g)$.
\end{lemma}
\begin{proof}
  We argue by contradiction. First recall that the volume of balls with sufficiently small radius are bounded from below in terms of their radius and the Sobolev constant. By a theorem of Croke \cite{Croke1980}, the Sobolev constant on a surface can be bounded in terms of the volume and the injectivity radius of the surface. Now suppose that $(M,g_n)$ is a sequence of complete metrics with $\inj(M, g_n) \geq \epsilon$ and $\Vol(M, g_n) < V$ and $\diam(M, g_n) = D_n \to \infty$. Let $x_n, y_n$ be such that $d(x_n, y_n) = D_n$. Then there exists a minimal geodesic $\gamma_n : [0,D_n] \to M$ such that $\gamma_n(0) = x_n$ and $\gamma_n(D_n) = y_n$. Because the Sobolev constant $C_S$ is bounded, there exists $r > 0$ and $\nu > 0$, such that the volume of any ball $B_x(r) \subset (M, g_n)$ of radius $r$ is greater than $\nu$. (See for instance lemma 2.2 in \cite{Hebey1999}.) We can assume $r < \epsilon$.
  The balls
  $$B_{\gamma_n(2kr)}(r) \text{ for } 0 \leq k \leq \left\lfloor\frac{D_n}{2r}\right\rfloor$$
  are pairwise disjoint. Hence
  $$\Vol(M, g_n) \geq \sum_{0 \leq k \leq \left\lfloor\frac{D_n}{2r}\right\rfloor} \int_{B_{\gamma_n(2kr)}(r)} \vol_{g_n} \geq \left\lfloor\frac{D_n}{2r}\right\rfloor \nu.$$
  Since $D_n \to \infty$, this implies
  $$\Vol(M, g_n)  \to \infty.$$
  This is a contradiction to our assumptions, and hence proves the lemma.
\end{proof}
\begin{proof}[Proof of theorem \ref{GeomCtrl}]
  Suppose $g_t$, $t \in (0,T)$ is a family satisfying the conditions of the theorem. By the uniformisation theorem, there exists a family of constant curvature metrics $\hat g_t$ of volume $1$ and a family of conformal factors $u_t$, such that
  $$g_t = e^{\hat 2u_t} \hat g_t.$$
  By the apriori estimate in \ref{APE}, there exists a uniform bound
  $$\|\hat u_t\|_{C^0} \leq C.$$
  By the previous lemma there is also a uniform bound on the diameters of $g_t$. Together with the uniform bound on the conformal factors $u_t$, we thus also obtain a uniform bound on the diameters of $\hat g_t$. This implies that the injectivity radii of the constant curvature metrics $\bar g_t$ are also bounded below. Now using the lemma \ref{Pullback}, we obtain a family of diffeomorphisms $f_t$, such that $f_t^* \bar g_t = \hat g_t$ satisfies $\partial_t \bar g_t \in \mc{H}_{\bar g_t}$. Letting $u_t = \hat u_t \circ f_t^{-1}$, we have
  $$g_t = f_t^* (e^{2 u_t}\bar g_t).$$
  This implies in particular that
  $$\|u_t\|_{C^0} \leq C.$$
  The curvature equation is
  $$R_{g_t} = e^{-2u_t} ( \Delta_{\bar g_t} u_t + R_{\bar g}),$$
  or equivalently
  $$\Delta_{\bar g_t} u_t = R_{g_t} e^{2u_t} - R_{\bar g}.$$
  By assumption the right hand side is bounded in $L^2(M, g_t)$. In particular, we can apply the Calderon--Zygmund inequality to $u_t - \int_M u_t \vol_{g_t}$ to obtain
  $$\|u_t\|_{H^2(M, g_t)} \leq C.$$
  This finishes the proof.
\end{proof}
The following theorem for horizontal curves of constant curvature metrics by Rupflin and Topping can be seen as an analogoue of the Mumford compactness theorem. Notice that in addition to control on the injectivity radius we also need control on the velocity of the curve (in the very weak $L^2$ norm).
\begin{thm}[\cite{Rupflin2016}, Lemma 2.6]
  \label{HorCurveCtrl}
  Suppose $g_t$, $t \in (0,T)$ is a family of constant curvature metrics with $\partial_t g_t \in \mc{H}_{g_t}$ for every $t$. Then there exists a constant $C > 0$ depending only on the genus of $M$ and $k \in \N$, such that
  $$\| \partial_t g_t\|_{C^k(M, g_t)} \leq C \frac{1}{\inj(M, g_t)^{1/2}} \|\partial_t g_t\|_{L^2(M, g_t)}.$$
\end{thm}
We will obtain a horizontal curve $\bar g_t$ from a family $g_t$ as in theorem \ref{GeomCtrl} above. The following lemma will show how control on the velocity of $g_t$ can be used to control the velocity of $\bar g_t$, which is necessary to apply theorem \ref{HorCurveCtrl}.
\begin{lemma}
  \label{VelCtrl}
  Assume $g_t \in \mc{M}$ is a smooth family of metrics on $\halfopen{0}{T}$ and let $\bar g_t \in \mc{M}^{cc}$ be the corresponding horizontal family of constant curvature metrics, i.e.
  $$\bar g_t = f_t^*( e^{-2u_t} g_t).$$
  If $u_t$ and $\|\partial_t g_t\|_{L^2(M, g_t)}$ are bounded on $\halfopen{0}{T}$, then so is $\|\partial_t \bar g_t\|_{L^2(M, \bar g_t)}$.
\end{lemma}
\begin{proof}
  Denote by $\hat g_t = e^{-2u_t} g_t$. Then
  $$\partial_t \hat g_t = -2 \partial_t u_t \hat g_t + e^{-2u_t} \partial_t g_t.$$
  The constant curvature condition on $\hat g_t$ implies
  $$\partial_t u_t = \frac{1}{4} \tr_{g_t} \partial_t g_t.$$
  In particular, if $u_t$ and $\|\partial_t g_t\|_{L^2(M, g_t)}$ is uniformly bounded, then so is $\|\partial_t \hat g_t\|_{L^2(M, \hat g_t)}$. As in the derivation of the split flow equations we find that
  $$\delta_{\bar g_t} \delta_{\bar g_t}^* X_t^{\flat} = \delta_{\bar g_t} (f_t^* \partial_t \hat g_t),$$
  where $X_t$ is the time dependent vector field generating $f_t$. Applying elliptic regularity yields
  $$\|X\|_{W^{1,2}(M, \bar g_t)} \leq C \|f_t^* \partial_t \hat g_t\|_{L^2(M, \bar g_t)} = C \|\partial_t \hat g_t\|_{L^2(M, \hat g_t)}.$$
  By assumption, the last term is uniformly bounded in $t$. The following calculation then implies
  \begin{align*}
    \|\partial_t \bar g_t\|_{L^2(M, \bar g_t)} & = \|\partial_t (f_t^* \hat g_t)\|_{L^2(M, \bar g_t)} \\
    & = \|f_t^* \mc{L}_{X_t} \hat g_t + f_t^* \partial_t \hat g_t \|_{L^2(M, f_t^* \hat g_t)} \\
    & \leq \|f_t^* \mc{L}_{X_t} \hat g_t\|_{L^2(M, f_t^* \hat g_t)} + \|f_t^* \partial_t \hat g_t\|_{L^2(M, f_t^* \hat g_t)} \\
    & = \|\delta_{\hat g_t}^* X_t^{\flat}\| + \|\partial_t \hat g_t\|_{L^2(M, \hat g_t)}
  \end{align*}
  that $\partial_t \bar g_t$ is also bounded uniformly in in $L^2(M, \bar g_t)$.
\end{proof}
A consequence of lemma \ref{VelCtrl} and theorem \ref{HorCurveCtrl} is that under the conditions of theorem \ref{GeomCtrl} the horizontal curve $\bar g_t$ is uniformly controlled, if $\|\partial_t g_t\|_{L^2(M,g_t)}$ remains bounded.
\section{The blow up criteria}
\label{blowup}
In this section we derive the blow up criteria from the introduction. The structure of the proofs is the same for the different flows: first, the equations for the split flow are derived for the given flow. We then check that the assumptions allow us to apply theorem \ref{GeomCtrl} to the resulting families of metrics and conformal factors. Finally, we use the parabolic structure of the equation to show that the conditions on the flow in fact imply that the solution remains uniformly bounded in $C^k$ on the existence interval. It turns out that the spinor flow is much more complicated to handle than the harmonic Ricci flow. This is due to the dependence of the spinor bundle on the tangent bundle.
\subsection{The harmonic Ricci flow}
Suppose $M$ is a closed surface with $\chi(M) \leq 0$. Suppose $(N, h)$ is a closed Riemannian manifold. A family of metrics $g_t \in \mc{M}$ and maps $\phi_t : M \to N$ satisfies the (volume normalised) harmonic Ricci flow, if
$$\partial_t g_t = T(\phi_t, g_t),$$
$$\partial_t \phi_t = \tau_g(\phi_t),$$
where $\tau_g(\phi)$ is the tension field of $\phi$ with respect to the metric $g$ and $T$ is defined by
$$T(\phi, g) = -2 \Ric_g + 2 \alpha d\phi \otimes d\phi + \left( \fint \tr_g (\Ric_g - \alpha d\phi \otimes d\phi) \right) g,$$
where $\alpha > 0$ is a time dependent coupling constant. In the metric component the flow is a modification of the Ricci flow and in the map component it is the harmonic map heat flow with a time dependent metric. This flow has been introduced by Buzano (né Müller) \cite{Mueller2012} and has since been studied extensively by numerous authors. In particular, in dimension $2$ it was proven by Buzano and Rupflin in \cite{Buzano2017} that the flow does not form singularities, if the coupling constant $\alpha$ is larger than a certain constant depending on the geometry of $N$. For smaller $\alpha$ is is not known if the flow forms singularities in finite time. We prove the following blow up criterium for the harmonic Ricci flow, which is independent of the coupling function $\alpha$.
\begin{thm}
  Suppose $(g_t, \phi_t)$ satisfies the harmonic Ricci flow on an interval $\halfopen{0}{T}$. If
  $$\sup_{0 \leq t < T} \int_M |R_{g_t}|^{2+\epsilon} \vol_{g_t} + \int_M |d\phi_t|^{4+\epsilon} \vol_{g_t} < \infty$$
  for some $\epsilon>0$ and
  $$\inf_{0 \leq t < T} \inj(M, g_t) >  0,$$
  then the solution may be extended to a larger interval $\halfopen{0}{\tilde T}$ with $\tilde T > 0$.
\end{thm}
\begin{proof}
The proof will proceed by analyzing the split flow equations instead. We will show that the assumptions are enough to find uniform bounds on the conformal factor and on the energy density. This then implies by a standard bootstrapping argument that the solution is in fact smooth up to time $T$, i.e. we can define $(g_T, \phi_T)$. Restarting the flow with that initial condition gives the theorem.

The split flow satisfies the following equations (see \cite{Buzano2017}):
\begin{align*}
  \partial_t \bar g_t & = P^{\mc H}_{\bar g} ( e^{-2u} \mathring{T}(\phi, \bar g)),\\
  \partial_t u_t  & = \frac{1}{4} \tr_g (T(\phi, g)) + \frac{1}{2} \rho - du(X),\\
  \partial_t \phi_t & = \tau_g(\phi) - d\phi(X),
\end{align*}
where $P^{\mc{H}}_{\bar g}$ is the orthogonal projection of $\Gamma(\odot^2 T^*M)$ on $\mc{H}_{\bar g}$ and $\rho$ is a function satisfying
$$- \Delta_{\bar g} \rho - R_{\bar g} \rho = \delta_{\bar g} \delta_{\bar g} (e^{-2u} \mathring{T}(\phi, g)) \text{ with } \int_M \rho \vol_{\bar g} = 0$$
and $X$ is a vector field satisfying
$$\delta_{\bar g} \delta_{\bar g}^* X = - \delta_{\bar g}(e^{-2u} \mathring{T}(\phi, g) - \rho \bar g).$$

The bootstrapping is contained in the next lemma.
\begin{lemma}
  Suppose $(\bar g_t, u_t, \psi_t, f_t)$ satisfies the split harmonic Ricci flow equations with
  $$\inf_{0 \leq t < T} \inj(M, \bar g_t) > 0,$$
  $$\sup_{0 \leq t < T} \|u\|_{C^0} < \infty,$$
  $$\sup_{0 \leq t < T} \|d\phi\|_{C^0} < \infty,$$
  then $u, \phi \in C_k^{\alpha, \alpha/2}$ for every $k$ with norms depending on $k$ and the bounds on $u$, $d\phi$ and the injectivity radii.
\end{lemma}
\begin{proof}[Proof of the lemma]
  If
  $$\inf_{0 \leq t < T} \inj(M, \bar g_t) > 0,$$
  then all the metrics $\bar g_t$ are uniformly equivalent in any $C^k$ norm by the Mumford compactness theorem, in particular the elliptic estimates for the operators $\Delta_{\bar g_t}$ admit uniform constants.

  The equation
  $$\partial_t u_t  = \frac{1}{4} \tr_g (T(\phi, g)) + \frac{1}{2} \rho - du(X)$$
  can be expressed as
  $$\partial_t u_t + e^{-2u_t} \Delta_{\bar g_t} u_t = K_{\bar g} (1 - e^{-2u_t}) + \alpha ( e^{-2u_t} |d \phi_t|^2 -\bar E(t)) + \frac{1}{2} \rho - du(X),$$
  using the equation for the Gauß curvature under conformal change. The term $\bar E(t)$ is the average total energy of $\phi$, i.e.
  $$\bar E(t) = \fint |d\phi|_g^2 \vol_g.$$
  Notice that
  $$\partial_t + e^{-2u_t} \Delta_{\bar g_t}$$
  is a uniformly parabolic operator. We will check that the right hand side of the above equation is bounded in any $L^p$ space. This follows by assumption for the terms $K_{\bar g} (1 - e^{-2u_t})$  and $\alpha ( e^{-2u_t} |d \phi_t|^2 -\bar E(t))$. For the term $\rho$, observe that
  $$\mathring{T}(\phi, g) = 2 \alpha (d\phi \otimes d\phi - |d\phi|^2g)$$
  is uniformly bounded. Using $L^p$ regularity for the equation
  $$- \Delta_{\bar g} \rho - R_{\bar g} \rho = \delta_{\bar g} \delta_{\bar g} (e^{-2u} \mathring{T}(\phi, g))$$
  we obtain bounds for $\rho$ in any $L^p$ space. For $X$ the same argument gives bounds in $W^{1,p}$ for every $p$. The Krylov--Safonov parabolic estimate for
  $$\partial_t u_t + e^{-2u_t} \Delta_{\bar g_t} u_t$$
  yields $u \in C^{\alpha, \alpha/2}$.

  Turning now to the map component of the flow, we first rewrite the equation for $\phi$. We may assume $N$ is isometrically embedded in some $\R^d$. Then
  $$\tau_g(\phi) = -\Delta_{g} \phi + A(\phi)(d\phi, d\phi),$$
  where $A$ is the second fundamental form of $N$. Thus the evolution equation for $\phi$ can be rewritten as
  $$(\partial_t + \Delta_{g_t}) \phi_t = A(\phi_t)(d\phi_t, d\phi_t) - d\phi_t(X_t).$$
  Since $u\in C^{\alpha, \alpha/2}$, it follows that $\partial_t + \Delta_{g_t}$ is a uniformly parabolic operator. The righthand side is uniformly bounded in time and space, hence in particular uniformly bounded in time in every $L^p(M)$. Using parabolic regularity, we thus obtain
  $$\phi \in W^{2,1}_p \text{ for every } p < \infty.$$
  Sobolev embedding then yields $\phi \in C_1^{\beta, \beta/2}$ for all $\beta < 1$. This implies $X$ is also in $C^{\beta, \beta/2}$ for all $\beta < 1$. We conclude that $(\partial_t + \Delta_{g_t}) \phi_t \in C^{\beta, \beta/2}$. Schauder theory then says that in fact
  $$\phi \in C^{2+\beta, 1+\beta/2}.$$
  It is now a matter of repeatedly applying Schauder theory to obtain higher regularity for both $u$ and $\phi$.
\end{proof}
To prove the theorem we now check that the conditions of the theorem imply that the conditions of the lemma are also fulfilled. Applying theorem \ref{GeomCtrl} to $g_t$ implies
$$\inf_{0 \leq t < T} \inj(M, \bar g_t) > \epsilon > 0$$
and
$$\sup_{0 \leq t < T} \|u\|_{C^0} < \infty$$
follow from that theorem. Moreover, lemma \ref{VelCtrl} and hence theorem \ref{HorCurveCtrl} apply, because
$$\|\partial_t g_t\|_{L^2(M, g_t)} \leq \|R_{g_t}\|_{L^2(M, g_t)} + \alpha \|d\phi\|_{L^4(M, g_t)}.$$
Thus it only remains to be checked that
$$\sup_{0 \leq t < T} \|d\phi\|_{C^0} < \infty.$$
As in the lemma, we assume $N$ is isometrically embedded in some $\R^d$ and we have the evolution equation
$$(\partial_t + \Delta_{g_t}) \phi_t = A(\phi_t)(d\phi_t, d\phi_t) - d\phi_t(X_t)$$
for $\phi$. Since $d\phi \in L^{4+2 \epsilon}$, the term $A(\phi)(d\phi_t, d\phi_t)$ is in $L^{2+\epsilon}$ since $A(\phi)$ is a bilinear form and since we assumed $N$ to be compact, we conclude
$$A(x)(V,V) \leq C |V|^2$$
for some constant $C$ independent of $x$. To see that $d\phi(X) \in L^{2+\epsilon}$, we first note that $T(g,\phi) \in L^{2+\epsilon}$. This implies by elliptic regularity that $\rho \in L^{2+\epsilon}$ as well. This in turn implies that $X \in W^{1,2+\epsilon}$. Since $W^{1,2+\epsilon} \hookrightarrow L^{\infty}$, we obtain that $d\phi(X) \in L^{4+2\epsilon}$, in particular $d\phi(X) \in L^{2+\epsilon}$. Thus
$$(\partial_t + \Delta_{g_t} ) \phi_t \in L^{2+\epsilon}.$$
Since $\partial_t + \Delta_{g_t}$ is a uniformly parabolic operator, this implies $\phi \in W^{2,1}_{2+\epsilon}(M \times [0,T])$. The Sobolev embedding theorem for $W^{2,1}_{2+\epsilon}$ then says that $\phi \in C^{\alpha,\alpha/2}$ with
$$\alpha = 2 - \frac{4}{2+\epsilon} = \frac{2\epsilon}{2+\epsilon}.$$
This proves the theorem.
\end{proof}
%{\em Remark.} It is quite plausible that the theorem also holds for $\epsilon = 0$, since for the harmonic heat flow the exponent $4$ is in fact sufficient, see \cite{Struwe1985}. Since our main purpose here is to illustrate the technique, we do not pursue this.

\subsection{The spinor flow}
Suppose now $M$ is a closed surface, $\chi(M) \leq 0$ and that we have chosen a topological spin structure on $M$. Consider the set
$$\mc{N} = \{(g, \varphi) : g \in \mc{M}, \varphi \in \Gamma(\Sigma_g M), |\varphi| \equiv 1\}.$$
The spinor flow is the gradient flow of the functional
$$\mc{E} : \mc{N} \to \R$$
$$\mc{E}(g, \varphi) = \frac{1}{2} \int_M |\nabla^g \varphi|^2 \vol_g$$
with respect to the natural $L^2$ norm on
$$T_{g,\varphi} \mc{N} = \Gamma(\odot^2 T^* M \oplus \Sigma_g M^{\perp \varphi} ).$$
For more details on this flow, see \cite{Ammann2016}. For the meaning of the energy functional in dimension $2$, see \cite{Ammann2016a}. Denote the negative gradient by
$$Q(g,\varphi) = (Q_1(g,\varphi), Q_2(g,\varphi)) = -\grad \mc{E}(g,\varphi),$$
where $Q_1$ is the metric component and $Q_2$ is the spinorial component. Then the spinor flow equations read
$$\partial_t g_t = Q_1(g_t, \varphi_t),$$
$$\partial_t \varphi_t = Q_2(g_t, \varphi_t).$$
The negative gradient $Q$ can be computed to be
$$Q_1(g,\varphi) = -\frac{1}{4} |\nabla^g \varphi|^2 g - \frac{1}{4} \divgc_g T_{g, \varphi} + \frac{1}{2} \langle \nabla^g \varphi \otimes \nabla^g \varphi \rangle$$
$$Q_2(g,\varphi) = -\nabla^{g*} \nabla^g \varphi + |\nabla^g \varphi|^2 \varphi.$$
The $3$ tensor $T_{g,\varphi}$ is defined to be
$$T_{g,\varphi}(X,Y,Z) = \sym_{Y,Z} \langle X \cdot Y \cdot \varphi, \nabla^g_Z \varphi \rangle,$$
where $\sym_{Y,Z}$ denotes the symmetrization in $Y$ and $Z$. The symmetric $2$ tensor $\langle \nabla^g \varphi \otimes \nabla^g \varphi \rangle$ is defined by the formula
$$\langle \nabla^g \varphi \otimes \nabla^g \varphi \rangle(X,Y) = \langle \nabla^g_X \varphi , \nabla^g_Y \varphi \rangle.$$
The trace part of $Q_1$ is given by the following important formula:
\begin{align}
  \tr_g Q_1(g, \varphi) & = -\frac{1}{4} \left( \langle D_g^2 \varphi, \varphi \rangle - |D_g \varphi|^2 \right) + \frac{2-n}{4} |\nabla^g \varphi|^2
\end{align}
where $n$ is the dimension of the manifold $M$. By the Weitzenböck formula and the constraint $|\varphi|^2 = 1$, we also have
$$\langle D_g^2 \varphi, \varphi \rangle = |\nabla^g \varphi|^2 + \frac{1}{4} R_g.$$
In particular, on a surface we have
\begin{align}
  \tr_g Q_1(g, \varphi) & = -\frac{1}{16} R_g - \frac{1}{4} |\nabla^g \varphi|^2 + \frac{1}{4} |D_g \varphi|^2.
\end{align}
In the following we decompose the spinor flow on surfaces using the framework from section \ref{splitting}. Assume $(\tilde g_t, \tilde \varphi_t)$ solves the spinor flow equation on a surface $M$, i.e.
$$\partial_t \tilde g_t = -\frac{1}{4} |\nabla^{\tilde g} \tilde \varphi|^2 \tilde g - \frac{1}{4} \divgc_{ \tilde g} T_{\tilde g,\tilde \varphi} + \frac{1}{2} \langle \nabla^{\tilde g} \tilde \varphi \otimes \nabla^{\tilde g} \tilde \varphi \rangle$$
$$\partial_t \tilde \varphi_t = -\nabla^{\tilde g*} \nabla^{\tilde g} \tilde \varphi + |\nabla^{\tilde g} \tilde \varphi|^2 \tilde \varphi.$$
Now assume that $(\bar g_t, u_t, \varphi_t)$ is the corresponding split flow. We also denote $g_t = e^{2u_t} \bar g_t$. Recalling that
$$\tr_g Q_1(g,\varphi) = -\frac{1}{4} \left( \langle D_g^2 \varphi, \varphi \rangle - |D_g \varphi|^2 \right)$$
and
$$\tilde{\mc{L}}_X \varphi = \nabla^g_X \varphi - \frac{1}{4} d X^{\flat} \cdot \varphi,$$
we obtain from proposition \ref{SFE} the following evolution equations for the split flow
$$\partial_t \bar g_t = P_{\bar g_t}\left( e^{-2u_t} \mathring{Q}_1(g_t, \varphi_t) \right)$$
$$\partial_t u_t = -\frac{1}{16} \left( \langle D_g^2 \varphi, \varphi \rangle - |D_g \varphi|^2\right) - Xu - \frac{1}{2} \rho$$
$$\partial_t \varphi_t = - \nabla^{g*} \nabla^g \varphi + |\nabla^g \varphi|^2 \varphi + \nabla^g_X \varphi - \frac{1}{4} dX^{\flat} \cdot \varphi,$$
where the vector field $X$ and the function $\rho$ are defined by the equations
$$\Delta_{\bar g} \rho + R_{\bar g} \rho = \delta_{\bar g} \delta_{\bar g} (e^{-2u} \mathring{Q}_1(g, \varphi))$$
$$\delta_{\bar g} \delta_{\bar g}^* X^{\flat} = - \delta_{\bar g} (e^{-2u} \mathring{Q}_1(g, \varphi) + \rho \bar g).$$
In fact, more detailed computations yield that $(\bar g_t, u_t, \varphi_t, \rho_t, X_t)$ satisfy the following system of equations. The detailed calculations are lengthy and can be found in the thesis of the author, \cite{Schiemanowski2018}. It should be remarked that in the calculation of $\rho$ the Bianchi identity $\delta_g Q_1(g, \varphi)$ plays an important role.
\begin{prop}
  \label{DSFE}
  The split flow satisfies
  \begin{subequations}
    \begin{align}
      & \partial_t u_t + \frac{1}{32} (1-R_{\bar g}) e^{-2 u_t} \Delta_{\bar g_t} u_t = -\frac{1}{16} (1 - R_{\bar g}) \left(\frac{1}{4} R_{\bar g} + |\nabla^g \varphi|^2 - |D_g \varphi|^2\right) - Xu - \frac{1}{2} \tilde \rho \label{CFSF} \\
      &  \partial_t \varphi_t + \frac{1}{32} e^{-2u_t} \nabla^{\bar g *} \nabla^{\bar g} \varphi = e^{-2u_t} \psi + \nabla^{\bar g}_X \varphi - \frac{1}{4} d X^{\flat} \cdot \varphi \label{SFSF} \\
      & \psi = -\grad_{\bar g} u \cdot D_{\bar g} \varphi - \nabla^{\bar g}_{\grad_{\bar g} u} \varphi - |\nabla^{\bar g} \varphi|^2 \varphi + \langle \grad_{\bar g} u \cdot D_{\bar g} \varphi, \varphi \rangle \varphi \label{STSF}\\
      & \rho = \tilde \rho + \frac{1}{2} \tr_g Q_1(g,\varphi) \label{RhoSF} \\
      & \Delta_{\bar g} \tilde \rho + R_{\bar g} \tilde \rho = -\frac{1}{2} R_{\bar g} \tr_g Q_1(g, \varphi) + \delta_{\bar g} \mathring{Q}_1(g, \varphi)(du, \cdot) \label{RhoTSF} \\
      & \delta_{\bar g} \delta_{\bar g}^* X^{\flat} = \delta_{\bar g} (\mathring{Q}_1(g, \varphi) + \rho \bar g) \label{XSF} \\
      & \delta_{\bar g} \mathring{Q}_1(g, \varphi)) = \frac{1}{2} e^{2u} d\tr_g Q_1(g, \varphi) \label{DQ1SF}
    \end{align}
  \end{subequations}
\end{prop}
Similarly to the previous section, we first prove a lemma stating that it is enough to show $C^{\alpha, \alpha/2}_2$ regularity of $u_t$ and $\varphi_t$ to obtain a uniformly smooth solution.
\begin{lemma}
  Suppose $(\bar g_t, u_t, \varphi_t)$, $t \in [0,T]$ is a solution of the split spinor flow equations on $M$ and suppose
  $$\inj(\bar g_t) > \epsilon > 0,$$
  $$\|u\|_{C_2^{\alpha,\alpha/2}(M \times [0, T])} \leq C$$
  and
  $$\|\varphi\|_{C_2^{\alpha,\alpha/2}(M \times [0, T])} \leq C.$$
  Then all higher space and time derivatives of $(g_t, \varphi_t)$ can be bounded in terms of $C$.
\end{lemma}
\begin{proof}
  By theorem \ref{HorCurveCtrl} the curve $\bar g_t$ admits uniform estimates in every $C^k$ norm. In particular, the Laplacians $\Delta_{\bar g_t}$ and the spin Laplacians $\nabla^{\bar g_t *} \nabla^{\bar g_t}$ are all equivalent in the sense that their coefficients admit a uniform bound in every $C^k$ norm.
  The evolution equation of $u_t$ can be rewritten as
  $$\partial_t u_t + \frac{1}{32} e^{-2u_t} \Delta_{\bar g_t} u_t = -\frac{1}{16} \left(\frac{1}{4} R_{\bar g} + |\nabla^g \varphi|^2 - |D_g \varphi|^2\right) - Xu - \frac{1}{2} \rho.$$
  Since
  $$\rho = \tilde \rho + \frac{1}{2} R_{\bar g} \tr_g Q_1(g, \varphi) = \tilde \rho - R_{\bar g} \left( \frac{1}{32} R_g + \frac{1}{8} |\nabla^g \varphi|^2 - \frac{1}{8} |D_g \varphi|^2\right)$$
  and $R_g = e^{-2u} (2 \Delta_{\bar g} u + R_{\bar g})$, the evolution equation can be rewritten as
  \begin{align*}
    \partial_t u_t + \frac{1}{32} e^{-2u_t} \Delta_{\bar g_t} u_t & = -\frac{1}{16} \left(\frac{1}{4} R_{\bar g} + |\nabla^g \varphi|^2 - |D_g \varphi|^2\right) - Xu - \frac{1}{2} \tilde \rho \\
    & \qquad + R_{\bar g} \left( \frac{1}{32} R_g + \frac{1}{8} |\nabla^g \varphi|^2 - \frac{1}{8} |D_g \varphi|^2\right) \\
    & = -\frac{1}{16} (1 - R_{\bar g}) \left(\frac{1}{4} R_{\bar g} + |\nabla^g \varphi|^2 - |D_g \varphi|^2\right) - Xu - \frac{1}{2} \tilde \rho + R_{\bar g} \frac{1}{32} e^{-2u} \Delta_{\bar g} u
  \end{align*}
  or equivalently
  $$\partial_t u_t + \frac{1}{32} (1-R_{\bar g}) e^{-2 u_t} \Delta_{\bar g_t} u_t = -\frac{1}{16} (1 - R_{\bar g}) \left(\frac{1}{4} R_{\bar g} + |\nabla^g \varphi|^2 - |D_g \varphi|^2\right) - Xu - \frac{1}{2} \tilde \rho.$$
  Since $R_{\bar g} = 0$ or $8 \pi \chi(M) < 0$, it follows that the left hand side is a uniformly parabolic operator with $C_2^{\alpha,\alpha/2}$ coefficients. To gain an improvement in regularity we aim to show that the right hand side is in $C_1^{\alpha,\alpha/2}$. Because $\varphi \in C_2^{\alpha, \alpha/2}$, we conclude $|\nabla^g \varphi|^2 - |D_g \varphi|^2 \in C^{\alpha,\alpha/2}_1$. Since $R_{\bar g}$ is constant, this implies that the bracketed term is in $C_1^{\alpha,\alpha/2}$. It remains to show $\tilde \rho$ and $Xu$ are in $C_1^{\alpha,\alpha/2}$. The function $\tilde \rho$ satisfies the elliptic equation
  $$\Delta_{\bar g} \tilde \rho + R_{\bar g} \tilde \rho = -\frac{1}{2} R_{\bar g} \tr_g Q_1(g, \varphi) + \delta_{\bar g} \mathring{Q}_1(g, \varphi)(du, \cdot).$$
  By assumption $Q_1(g, \varphi) \in C^{\alpha, \alpha/2}$. Thus from Schauder estimates for equations where the data is in divergence form, we have that
  $$\tilde \rho \in C_1^{\alpha,\alpha/2}.$$
  The same holds for $X$, i.e. since
  $$\delta_{\bar g} \delta_{\bar g}^* X^{\flat} = \delta_{\bar g} (\mathring{Q}_1(g, \varphi) + \rho \bar g)$$
  is an elliptic operator and $\mathring{Q}_1(g, \varphi) \in C^{\alpha, \alpha/2}$, it follows that
  $$X \in C^{\alpha,\alpha/2}_1.$$
  We conclude that indeed
  $$\partial_t u_t + \frac{1}{32} (1-R_{\bar g}) \Delta_{\bar g_t} u_t \in C^{\alpha, \alpha/2}_1$$
  and hence by Schauder estimates
  $$u \in C^{2+\alpha, 1+\alpha/2}_1,$$
  which implies $u \in C_3^{\alpha, \alpha/2}$. Now we turn to the equation for the spinorial part $\varphi$. The equation
  $$\partial_t \varphi_t = - \nabla^{g*} \nabla^g \varphi + |\nabla^g \varphi|^2 \varphi + \nabla^g_X \varphi - \frac{1}{4} dX^{\flat} \cdot \varphi$$
  can be rewritten as
    \begin{align*}
    \partial_t \varphi_t + \frac{1}{32} e^{-2u_t} \nabla^{\bar g *} \nabla^{\bar g} \varphi & = e^{-2u_t} \left( -\grad_{\bar g} u \cdot D_{\bar g} \varphi - \nabla^{\bar g}_{\grad_{\bar g} u} \varphi - |\nabla^{\bar g} \varphi|^2 \varphi + \langle \grad_{\bar g} u \cdot D_{\bar g} \varphi, \varphi \rangle \varphi \right)\\
    & + \nabla^{\bar g}_X \varphi - \frac{1}{4} d X^{\flat} \cdot \varphi.
    \end{align*}
    The left hand side is a parabolic operator with $C_2^{\alpha,\alpha/2}$ coefficients. On the other hand the right hand side consists of terms, which depend on the first derivatives of $g$ and $\varphi$, with the exception of the terms involving $X$. These terms are seen to be $C_1^{\alpha,\alpha/2}$. The term $\nabla^{\bar g}_X \varphi$ is $C_1^{\alpha,\alpha/2}$, because $\varphi \in C_2^{\alpha,\alpha/2}$ and $X \in C^{\alpha,\alpha/2}_1$. The term $d X^{\flat} \cdot \varphi$ is more delicate, because a derivative of $X$ is involved. Thus we need to show that $X \in C_2^{\alpha,\alpha/2}$. Then we can conclude that $d X^{\flat} \in C_1^{\alpha, \alpha/2}$. However, we already know $u \in C_3^{\alpha,\alpha/2}$. Since $X$ is the solution of
    $$\delta_{\bar g} \delta_{\bar g}^* X^{\flat} = \delta_{\bar g} (\mathring{Q}_1(g, \varphi) + \rho \bar g),$$
    and $\rho \in C_1^{\alpha,\alpha/2}$, it remains to be seen that $\delta_{\bar g} \mathring{Q}_1(g, \varphi) \in C^{\alpha,\alpha/2}$ to conclude that $X \in C_2^{\alpha,\alpha/2}$. This follows from Schauder theory, because the term
    $$\delta_{\bar g} \mathring{Q}_1(g, \varphi)) = \frac{1}{2} e^{2u} d\tr_g Q_1(g, \varphi)$$
    is in $C^{\alpha,\alpha/2}$. This can be seen from the formula
    $$\tr_g Q_1(g, \varphi) = -\frac{1}{4} (R_g/4 + |\nabla^g \varphi|^2 - |D_g \varphi|^2).$$
    The only second order term is the curvature $R_g$, which depends only on $u$, and hence is in $C_1^{\alpha,\alpha/2}$.

    For the higher regularity, we can repeat this line of argument.
\end{proof}
We now proceed to the proof of the blow up criterium from the introduction.
\begin{proof}[Proof of theorem \ref{IntBC}]
  We assume for the solution $(g_t, \varphi_t)$ of the spinor flow on $\halfopen{0}{T}$ that
  $$\sup_{0 < t < T} \int_M |\nabla^2 \varphi_t|^q \vol_{g_t} < \infty$$
  for some $q > 4$ and
  $$\inf_{0 < t < T} \inj(M, g_t) > 0.$$
  The second covariant derivative of $\varphi_t$ can be orthogonally decomposed into a symmetric and an antisymmetric part:
  $$\nabla^g \nabla^g \varphi = (\nabla^g \nabla^g \varphi)^{sym} + (\nabla^g \nabla^g \varphi)^{asym}.$$
  The antisymmetric part is the curvature of the spin connection. Since the curvature of the spin connection on a surface is given by
  $$R^g(X,Y)\varphi = \frac{R_g}{4} g(X,Y) \omega \cdot \varphi,$$
  it follows that for a unit spinor
  $$|(\nabla^g \nabla^g \varphi)^{asym} \varphi|^2 = \frac{1}{8} R_g^2.$$
  Consequently, a bound on $\int_M |\nabla^g \nabla^g \varphi|^q \vol_g$ implies a bound on $\int_M |R_g|^q \vol_g$.

  Now suppose $(\tilde g_t, \tilde \varphi_t)$ is a smooth solution of the spinor flow on the interval $(0,T)$ satisfying
  $$\sup_{0 < t < T} \int_M |\nabla^{\tilde g} \nabla^{\tilde g} \tilde \varphi|^q \vol_{\tilde g_t} < \infty$$
  and
  $$\inf_{0 < t < T} \inj(\tilde g_t) > 0.$$
  Now consider the corresponding split flow $(\bar g_t, u_t, \varphi_t)$ and denote $g_t = e^{2u_t} \bar g_t$. Notice that the bounds are diffeomorphism invariant, so we get the same conditions for $(g_t, \varphi_t)$. Theorem \ref{GeomCtrl} applies to the family $g_t$ and we obtain that $\bar g_t$ has injectivity radius bounded from below and that $u_t$ is bounded in $C^0(M, \check g)$ for any fixed metric $\check g$. Thus we can apply $L^p$ theory to the curvature equation to conclude that $u \in W^{2,q}$, and in particular in $C^{1,\alpha}$ by Sobolev embedding.

  The curve of metrics $\bar g_t$ is smooth on $\halfopen{0}{T}$. Since the injectivity radii of $\bar g_t$ are bounded from below and the velocity of the metric is bounded above (since the spinor flow is a gradient flow), we obtain from theorems \ref{GeomCtrl} and \ref{HorCurveCtrl} that $\bar g_t$ extends to a continuous curve $\halfclosed{0}{T}$ and that the metrics $\bar g_t$, $0 < t \leq T$, are all uniformly equivalent in $\CI$.

  If we show that $u, \varphi \in C_2^{\alpha, \alpha/2}$, then by the lemma we have uniform estimates of $u_t, \varphi_t$ in any $C^k$ norm and thus we may pass to a smooth limit as $t \to T$ and the flow can be restarted at time $T$, yielding a solution on $\halfopen{0}{\tilde T}$ for some $\tilde T > T$ by short time existence.

  In the following we indicate the steps we will take. For this, $r$ and $\alpha$ denote constants which may change from line to line. Recall the split flow equations from proposition \ref{DSFE}. We first show that $\tilde \rho_t$ and $X_t$ are uniformly bounded in $W^{1,r}$. This will imply
  $$\partial_t u_t + \frac{1}{32} (1- R_{\bar g}) e^{-2 u_t} \Delta_{\bar g_t} u_t \in W^{1,r}.$$
  This implies by parabolic regularity
  $$u, du \in W^{2,1}_r,$$
  which implies
  $$u, du \in C_1^{\alpha, \alpha/2},$$
  i.e. $u \in C_2^{\alpha,\alpha/2}$. From this we conclude that $X \in C_1^{\alpha, \alpha/2}$.
  We will then first show that
  $$\partial_t \varphi_t + \frac{1}{32} e^{-2u_t} \nabla^{\bar g *} \nabla^{\bar g} \varphi \in L^r.$$
  This implies
  $$\varphi \in W_r^{2,1} \subset C^{\alpha, \alpha/2}.$$
  From this we can conclude
  $$\partial_t \varphi_t + \frac{1}{32} e^{-2u_t} \nabla^{\bar g *} \nabla^{\bar g} \varphi \in C^{\alpha, \alpha/2},$$
  and finally by parabolic Schauder theory
  $$\varphi \in C^{2+\alpha,1+\alpha/2}.$$
  \underline{Step 1: $\tilde \rho \in W^{1,q'}$ for every $q' < q$:}\\
  Recall that $\tilde \rho$ satisfies equation \ref{RhoTSF}:
  $$\Delta_{\bar g} \tilde \rho + R_{\bar g} \tilde \rho = -\frac{1}{2} R_{\bar g} \tr_g Q_1(g, \varphi) + \delta_{\bar g} \mathring{Q}_1(g, \varphi)(du, \cdot).$$
  Since $\tr_g Q_1(g,\varphi) = -\frac{1}{4} ( R_g/4 + |\nabla^g \varphi|^2 - |D_g \varphi|^2)$, it follows that $\tr_g Q_1 \in L^q$, since $R_g$ is and the other two terms are by Poincaré inequality. Since $Q_1$ has the structural form $\nabla^g \nabla^g \varphi \ast \varphi + \nabla^g \varphi \ast \nabla^g \varphi$, it follows that $\mathring{Q}_1 \in L^q$ and since $du \in W^{1,q}$, it follows that
  $$\mathring{Q}_1(g,\varphi)(du, \cdot) \in L^{q'} \text{ for every } q' < q.$$
  By elliptic theory it follows that
  $$\tilde \rho \in W^{1,q'} \text{ for every} q' < q.$$
  \underline{Step 2: $X \in W^{1,q}$:}\\
  By the previous step we know $\tilde \rho \in W^{1, q'}$ and hence by Sobolev inequality $\tilde \rho \in L^p$ for every $p$. By equation \ref{RhoSF}, it follows that $\rho \in L^q$. From the equation \ref{XSF}
  $$\delta_{\bar g} \delta_{\bar g}^* X^{\flat} = \delta_{\bar g} (\mathring{Q}_1(g, \varphi) + \rho \bar g),$$
  the claim then follows by elliptic theory.
  
  \underline{Step 3: $u, du \in W^{2,1}_{q'}$ for every $q' < q$:}\\
  First we note that we know $u$ is uniformly bounded. Thus by the Krylov--Safonov estimate for parabolic equations, it suffices to show that
  $$\partial_t u_t + \frac{1}{32} (1- R_{\bar g}) e^{-2 u_t} \Delta_{\bar g_t} u_t \in L^3,$$
  to conclude that $u$ is Hölder continuous both temporally and spatially. (We already knew that $u$ is Hölder continuous spatially from Sobolev embedding. The new information is the temporal continuity.) Thus we can apply the standard $L^p$ theory for equations with Hölder continuous coefficients. The claim will then follow from
  $$\partial_t u_t + \frac{1}{32} (1- R_{\bar g}) e^{-2 u_t} \Delta_{\bar g_t} u_t \in W^{1,q'}.$$
  That $u \in W^{2,1}_{q'}$ is then immediate from parabolic regularity. For $du \in W^{2,1}_{q'}$, notice that if the above term is in $W^{1,q'}$, then
  $$\partial_t du_t + \frac{1}{32} (1- R_{\bar g}) \left(-2 e^{-2u_t} (\Delta_{\bar g_t} u_t) du_t + e^{-2u_t} \nabla^{\bar g_t *} \nabla^{\bar g_t} du_t + \Ric^{\bar g_t}(\cdot, \grad_{\bar g_t} u) \right) \in L^{q'}.$$
  Since the term $e^{-2u_t} (\Delta_{\bar g_t} u_t) du_t$ is in $L^{q'}$ and the term $e^{-2u_t} \Ric^{\bar g_t}(\cdot, \grad_{\bar g_t} u)$ is in $L^p$ for every $p$, it follows that $du_t \in W^{2,1}_{q'}$ by parabolic regularity.
  
  We now show that $\partial_t u_t + \frac{1}{32} (1- R_{\bar g}) e^{-2 u_t} \Delta_{\bar g_t} u_t \in W^{1,q'}$. Equation \ref{CFSF} says that this term is equal to
  $$-\frac{1}{16} (1 - R_{\bar g}) \left(\frac{1}{4} R_{\bar g} + |\nabla^g \varphi|^2 - |D_g \varphi|^2\right) - Xu - \frac{1}{2} \tilde \rho.$$
  We know that $\nabla^g \varphi \in W^{1,q}(M,g)$ by Poincaré inequality. Hence
  $|\nabla^g \varphi|^2, |D_g \varphi|^2 \in W^{1,q/2}$. The term $R_{\bar g}$ is constant. Thus the bracketed term is in $W^{1, q/2}$. We have already seen $\tilde \rho \in W^{1,q'}$. Furthermore $Xu = du(X)$ is in $W^{1,q}$, since $du \in W^{1,q}$, $X \in W^{1,q}$ and $W^{1,q}$ is a Banach algebra for our choice of $q$. This proves the claim.

  \underline{Step 4: $\varphi \in W_q^{2,1}$:}\\
  Recall that $\varphi$ satisfies equation \ref{SFSF}, which says that
  $$\partial_t \varphi_t + \frac{1}{32} e^{-2u_t} \nabla^{\bar g *} \nabla^{\bar g} \varphi$$
  equals
  $$e^{-2u_t} \psi + \nabla^{\bar g}_X \varphi - \frac{1}{4} d X^{\flat} \cdot \varphi.$$
 Since $X \in W^{1,q}$, it follows that $dX^{\flat} \cdot \varphi \in L^q$. Furthermore, since $\nabla^g \varphi \in W^{1,q}$, it follows that $\nabla^{\bar g}_X \varphi \in W^{1,q/2} \subset L^q$. The term $\psi$ is given by
  $$-\grad_{\bar g} u \cdot D_{\bar g} \varphi - \nabla^{\bar g}_{\grad_{\bar g} u} \varphi - |\nabla^{\bar g} \varphi|^2 \varphi + \langle \grad_{\bar g} u \cdot D_{\bar g} \varphi, \varphi \rangle \varphi.$$
  Checking term by term, we also conclude that $\psi \in L^p$ for every $p$. Hence
  $$\partial_t \varphi_t + \frac{1}{32} e^{-2u_t} \nabla^{\bar g *} \nabla^{\bar g} \varphi \in L^q,$$
  and the claim follows by parabolic regularity theory.
  
  \underline{Step 5: $X \in C_1^{\alpha, \alpha/2}$ for $0 < \alpha < 1 - \frac{4}{q}$:}\\
  By steps 3 and 4 and Sobolev embedding, $u \in C_2^{\alpha,\alpha/2}$ and $\varphi \in C_1^{\alpha,\alpha/2}$ for every $\alpha < 1 - \frac{4}{q}$. Thus $\mathring{Q}_1(g,\varphi) + \rho \bar g \in C^{\alpha, \alpha/2}$. From this it follows by equation \ref{XSF} that $X \in C_1^{\alpha, \alpha/2}$ by Schauder theory for data in divergence form.

  \underline{Final step: $u \in C^{2+\alpha,1+\alpha/2}$ and $\varphi \in C^{2+\alpha,1+\alpha/2}$:}\\
  It now follows easily that the right hand side in equation \ref{CFSF} is in $C^{\alpha,\alpha/2}$. The claim $u \in C^{2+\alpha,1+\alpha/2}$ thus follows from parabolic Schauder theory. Likewise, it is easily seen that the right hand side of equation \ref{SFSF} is in $C^{\alpha,\alpha/2}$ and parabolic Schauder theory implies $\varphi \in C^{2+\alpha,1+\alpha/2}$.
\end{proof}

\begin{proof}[Proof of theorem \ref{UniBC}]
  The assumption is that
  $$\sup_{\substack{x \in M \\ 0 < t < T}}|\nabla^2 \varphi_t(x)| < \infty.$$
  Clearly, this condition implies that the integral of $|\nabla^2 \varphi|$ stays bounded in any $L^p$ norm. To conclude the theorem from theorem \ref{IntBC}, we need to show that the injectivity radius stays bounded from below. The bound on $|\nabla^2 \varphi|$ implies a pointwise bound on the curvature. Moreover, the spinor flow is volume preserving in two dimensions. Thus, if we have in addition a diameter bound on $g_t$, we can conclude that the injectivity radius stays bounded below. If $|\nabla^2 \varphi|$ stays bounded, then so does $|Q_1(g, \varphi)|$. Thus by integration along the flow the distance functions $d_{g_t}$ are uniformly equivalent, and hence the diameters remain bounded along the flow.
\end{proof}

\bibliographystyle{hplain}
\bibliography{lit}

\end{document}